\documentclass[11pt,a4paper,fleqn]{amsart}
\usepackage{amssymb}
\usepackage{hyperref}
\usepackage{natbib}

\newcommand{\ee}{\mathbb{E}}
\newcommand{\nn}{\mathbb{N}}
\newcommand{\pp}{\mathbb{P}}
\DeclareMathOperator{\Var}{Var}
\newcommand{\cn}{\mathcal{N}}

\newcommand{\dto}{\stackrel{d}{\to}}
\newcommand{\de}{\stackrel{d}{=}}
\newcommand{\toi}{\to\infty}
\newcommand{\as}[1]{\quad\text{as }#1\toi}
\newtheorem{thm}{Theorem}
\newtheorem{lem}[thm]{Lemma}
\theoremstyle{remark}
\newtheorem{rmk}{Remark }

\newcommand{\beq}{\begin{equation}}
\newcommand{\eeq}{\end{equation}}
\newcommand{\bt}{\begin{thm}}
\newcommand{\et}{\end{thm}}
\newcommand{\bl}{\begin{lem}}
\newcommand{\el}{\end{lem}}

\begin{document}
\bibliographystyle{plainnat}
\setcitestyle{numbers}
\title{Asymptotics for sums of a function of normalized independent sums}
\author{Kamil Marcin Kosi\'nski}
\email{k.m.kosinski@uva.nl}
\address{Wydzia{\l} Matematyki, Informatyki i Mechaniki, Uniwersytet Warszawski, Warszawa, Poland}
\curraddr{Korteweg-de Vries Institute for Mathematics, University of Amsterdam, P.O. Box 94248,
1090 GE Amsterdam, The Netherlands}

\date{Semptember 17, 2008}
\subjclass[2000]{Primary G0F05}
\keywords{Products of sums of iid rv's, Limit distribution, Central limit theorem, Lognormal distribution}

\begin{abstract}
We derive a central limit theorem for sums of a function of independent sums of independent and 
identically distributed random variables. In particular we show that previously known result
from Rempa{\l}a and Weso{\l}owski (Statist. Probab. Lett. 74 (2005) 129--138), which can be obtained by applying
the logarithm as the function, holds true under weaker assumptions.
\end{abstract}
\maketitle

\section{Introduction}
Throughout this paper let $(X_{k,n})_{k=1,\ldots,n}$; $n=1,2\ldots$ be a triangular array of
independent and identically distributed (iid) random variables (rv's) with the same distribution as $X$. 
Let us define the (mutually independent) partial sums $S_{n,n}=\sum_{k=1}^n X_{k,n}$. We will
work with real functions $f$ defined at least on an interval $I$ such that $\pp(X\in I)=1$.
We will also write $\log^+ x$ for $\log (x\vee 1)$.
\par The asymptotic behavior of the product of partial sums of a sequence of iid rv's 
has been studied in several papers (see e.g. \citet{LuiQi}
for a brief review). In particular it was shown in \citet{Remiweso} that
if $(X_n)$ is a sequence of iid positive square integrable rv's with $\ee X_1=\mu$,
$\Var(X_1)=\sigma^2>0$, then setting $S_n=\sum_{k=1}^n X_k$ and $\gamma=\sigma/\mu$ we have as $n\toi$
\[
  \left(\prod_{k=1}^n\frac{S_k}{k\mu}\right)^{\frac{1}{\gamma\sqrt{n}}}\dto e^{\sqrt{2}\cn},
\]
where $\dto$ stands for convergence in distribution and $\cn$ is a standard normal random variable.
This result was extended in \citet{Qi} and \citet{LuiQi} to a general limit theorem covering
the case when the underlying distribution is integrable and belongs to the domain of attraction of a stable law with index from 
the interval $[1,2]$.
\par This study brought an interest to the array case, where we no longer consider a sequence 
$(X_n)$ but a triangular array $(X_{k,n})$. In \citet{Remiweso2} the analogous result
was obtained, namely
\beq
\label{result:rw}
  \left(n^{\frac{\gamma^2}{2}}\prod_{k=1}^n\frac{S_{k,k}}{k\mu}\right)^{\frac{1}{\gamma\sqrt{\log n}}}\dto e^{\cn},
\eeq
 under the assumption $\ee |X|^p<\infty$ for some $p>2$.
\par The purpose of this paper is to show that the above result holds true under the assumption $\ee X^2(\log^+ |X|)^{1/2}<\infty$.
However, it is no longer true in general when only $\ee X^2<\infty$ is assumed.
We will show that under this assumption, different normalisation is needed.
Furthermore, we will set our discussion in a more general setting. It is straightforward that \eqref{result:rw}
is a simple corollary from
\[
  \frac{\sum_{k=1}^n f(S_{k,k}/k)-b_n}{a_n}\dto\cn,
\]
if one sets $f(x)=\log x$ and chooses the sequences $a_n$, $b_n$ properly.
\section{Main result}
\label{main}
\begin{thm}
\label{theorem1} Suppose that $\ee |X|^2<\infty$ and denote $\mu=\ee X$, $\sigma^2=\Var(X)$.
Let $f$ be a real function with bounded third derivative on some neighbourhood of $\mu$.
Then as $n\toi$ 
\[
 \frac{\sum_{k=1}^n f(S_{k,k}/k)-b_n}{a_n}\dto\sigma f'(\mu)\cn,
\]
where
\[
a_n=\sqrt{\log n}\,,\quad b_n= n f(\mu)+\frac{f''(\mu)}{2}\sum_{k=1}^n\frac{1}{k}\ee|X-\mu|^21_{\{|X-\mu|\le\sigma k\}}.
\]
\end{thm}
\begin{rmk}
If we strengthen the assumption of the square integrability of random variable $X$ to $\ee X^2(\log^+|X|)^{1/2}<\infty$,
then we can take the sequence $\tilde{b}_n=n f(\mu)+\frac{f''(\mu)\sigma^2}{2}\log n$ instead of $b_n$.
To see this we should show $\tilde{b}_n-b_n=o(\sqrt{\log n})$, which since
 $\log n-\sum_{k=1}^n\frac{1}{k}=O(1)$
is equivalent to
\[
  Q_n:=\sum_{k=1}^n\frac{1}{k}\left(\sigma^2-\ee|X-\mu|^21_{\{|X-\mu|\le\sigma k\}}\right)=o(\sqrt{\log n}).
\]
Observe that $Q_n$ is positive and
\begin{align*}
  Q_n&=
\sum_{k=1}^n\frac{1}{k}\ee|X-\mu|^21_{\{|X-\mu|>\sigma k\}}
=\ee|X-\mu|^2\sum_{\sigma k< |X-\mu|,\, k\le n}\frac{1}{k}\\
&\sim\ee|X-\mu|^2\log^+\left(n\wedge(|X-\mu|/\sigma)\right)=:\tilde{Q}_n.
\end{align*}
Therefore, if $\ee X^2(\log^+ |X|)^{1/2}<\infty$, we can use the Dominated Convergence Theorem and
infer that $\tilde{b}_n-b_n=o(\sqrt{\log n})$.
\end{rmk}
\begin{rmk}
On the other hand, if for some $\varepsilon>0$ we define a random variable $X$ by setting
$\pp(X=\pm k_n)=C/(2 k_n^2 n^2)$ and $\pp(X=0)=1-\sum_{n}\pp(|X|=k_n)$, 
where $k_n=e^{n^{2+\varepsilon}}$
and $C=6/\pi^2$. Then we simply have $\mu=0$, $\sigma^2=1$ and $\ee X^2(\log^+ |X|)^{1/2}=\infty$. 
Although, one can check that $\limsup_n Q_n/\sqrt{\log n}=\limsup_n \tilde{Q}_n/\sqrt{\log n}=\infty$,
which means that we cannot use $\tilde{b}_n$ in general.
\end{rmk}
Now let us take any positive (i.e. $I\subset(0,\infty)$), nondegenerate  random variable $X$ with $\ee X^2(\log^+ |X|)^{1/2}<\infty$
and $f(x)=\mu\log\left(x/\mu\right)$. Theorem \ref{theorem1} yields \eqref{result:rw}, that is 
the result from \citet{Remiweso2} however under weaker assumptions. Our argument shows that their result 
holds true even under the assumption of square integrability, although the normalizing sequences should be different.
Namely, instead of the term $n^{\frac{\gamma^2}{2}}$ in \eqref{result:rw} we should have
 $\exp(\frac{1}{2\mu^2}\sum_{k=1}^n\frac{1}{k}\ee|X-\mu|^21_{\{|X-\mu|\le\sigma k\}})$. 
\par The proof of Theorem \ref{theorem1} relies on the Taylor's expansion of a function $f$ in the neighbourhood of $\mu$. Linear term
in this expansion obeys a version of the classical Central Limit Theorem. Other terms are negligible 
mainly due to the Strong Law of Large Numbers (SLLN). This assertion will be made valid by a series of lemmas.
\bl
\label{lemma:CLT}
 Under the assumptions of Theorem \ref{theorem1} with $\sigma>0$
\[
 \frac{1}{\sigma\sqrt{\log n}}\sum_{k=1}^n \left(\frac{S_{k,k}}{k}-\mu\right)\dto\cn\as n. 
\]
\el
\begin{proof}
We may assume $\ee X=0$ and $\ee X^2=1$ by a simple normalization argument.
Since
\[
  \Var\left(\sum_{k=1}^n \frac{S_{k,k}}{k}\right)=\sum_{k=1}^n\frac{\Var(S_{k,k})}{k^2}=\sum_{k=1}^n\frac{1}{k}\sim \log n,
\]
then to complete the proof it is sufficient to show the Lindeberg condition for the array $(\frac{S_{k,k}}{k\sqrt{\log n}})_{k\le n}$, 
that is
\[
  \mathop{\text{\Large$\forall$}}_{r>0}\quad\frac{1}{\log n}\sum_{k=1}^n\ee \left(\frac{S_{k,k}}{k}\right)^2 1_{\{|S_{k,k}/k|>r \sqrt{\log n}\}}=o(1).
\]
Since $\{(S_{k,k}/\sqrt{k})^2\}$ is uniformly integrable,
\[
  \sup_{k\in\nn} \ee \left(\frac{S_{k,k}}{\sqrt{k}} \right)^2 1_{\{|S_{k,k}/\sqrt{k}|>r\sqrt{\log n}\}}\to0\as n,
\]
proving the Lindeberg condition.
\end{proof}
To establish the SLLN we will refer to \citet{Hsu} Law of Large Numbers (cf. \citet{Li}
for partial bibliographies and brief discussions).
\bl[Hsu-Robbins LLN]
\label{lemma:hsu-r}
For a sequence $(X_n)$ of iid rv's with $\ee X_1=0$ and $\ee X_1^2<\infty$ the series
\beq
\label{hsu-r:statement}
  \sum_{n=1}^\infty\pp(|S_n/n|>t)
\eeq
converges for every $t>0$.
\el
The condition \eqref{hsu-r:statement} implies $S_n/n\to 0$ a.s. under the Borel-Cantelli lemma. 
Moreover if $X_1$ in Lemma \ref{lemma:hsu-r} has the same distribution as $X$ in Theorem \ref{theorem1}, 
then $\pp(|S_n/n|>t)=\pp(|S_{n,n}/n|>t)$, i.e $S_{n,n}/n\to 0$ a.s. as well.
\par Before we proceed, we need some technical results derived from the elementary fact about the moments of sums of iid rv's
 (e.g., \citet[p. 23]{Hall}).
\bl[Rosenthal's inequality]
  \label{Rosenthal}
If $(X_n)$ is a sequence of iid rv's with the zero mean,
then for any $p\ge 2$
\[
  \ee|S_n|^p\leq C_p\left(n\ee |X_1|^p+n^{p/2}(\ee X_1^2)^{p/2}\right),
\]
where $C_p$ is a constant depending only on $p$.
\el
We will use this version of the Rosenthal's inequality to prove the following lemma, which
on the other hand will simplify a number of steps in the next lemma. The later will play a crucial role in the proof
of the main theorem.
\bl
\label{lemma:cor}
 Under the assumptions of Theorem \ref{theorem1}, for every $p>2$
\[
  \ee\sum_{k=1}^\infty\left(\frac{|\tilde{T}_k|}{k}\right)^p<\infty,
\]
where $\tilde{T}_k=\sum_{i=1}^k\left( X_{i,k}1_{\{|X_{i,k}|\le k\}}-\ee X_{i,k}1_{\{|X_{i,k}|\le k\}}\right)$.
\el
\begin{proof}
  Let $Y_k\de X_{i,k}1_{\{|X_{i,k}|\le k\}}-\ee X_{i,k}1_{\{|X_{i,k}|\le k\}}$, then by Lemma \ref{Rosenthal}
\begin{align*}
  \ee\sum_{k=1}^\infty\left(\frac{|\tilde{T}_k|}{k}\right)^p&\le 
C_p\sum_{k=1}^\infty\frac{1}{k^p}\left(k\ee |Y_k|^p+k^{p/2}(\ee Y_k^2)^{p/2}\right)\\
&\le C_p 2^p\left(\sum_{k=1}^\infty\frac{1}{k^{p-1}}\ee |X|^p1_{\{|X|\le k\}}+\sum_{k=1}^\infty\frac{1}{k^{p/2}}(\ee X^2)^{p/2}\right)\\
&= C_p2^p\left(\ee |X|^p\sum_{k\ge |X|}^\infty k^{1-p}+(\ee X^2)^{p/2}\sum_{k=1}^\infty\frac{1}{k^{p/2}}\right)\\
&\le C_p2^p\left( C \ee |X|^2+(\ee X^2)^{p/2}\sum_{k=1}^\infty\frac{1}{k^{p/2}}\right)<\infty,
\end{align*}
for some positive constant $C$.
\end{proof}
\bl
\label{lemma:remainder}
 Under the assumptions of Theorem \ref{theorem1} we have
\begin{align}
\label{remainder:1}
 \sum_{k=1}^n\left[\left(\frac{S_{k,k}-k\mu}{k}\right)^2-\frac{\ee|X-\mu|^21_{\{|X-\mu|\le\sigma k\}}}{k}\right]& 
=O_\pp(1),\\
\label{remainder:2}
\sum_{k=1}^n\left|\frac{S_{k,k}-k\mu}{k}\right|^3 & = O_\pp(1).
\end{align}
\el
\begin{proof} To simplify the notation we will write $S_k$ for $S_{k,k}$.
First note that to show \eqref{remainder:1} it is sufficient to prove that
\[
  \sum_{k=1}^n\left(\frac{S_k^2}{k^2}-\frac{\ee|X|^21_{\{|X|\le k\}}}{k}\right)=O_\pp(1),
\]
for a normalized random variable $X$.
Take any $\varepsilon>0$ and let $T_k:=\sum_{i=1}^k X_{i,k}1_{\{|X_{i,k}|\le k\}}$. 
Then $\sum_{k=1}^\infty\pp(S_k\ne T_k)\le \sum_{k=1}^\infty k\pp(|X|>k)<\infty$,
because $\ee X^2<\infty$. Hence we can take $R$ big 
enough that $\sum_{k=R}^\infty\pp(S_k\ne T_k)<\varepsilon/3$ and $M$ big enough that
\[
\pp\left(\left|\sum_{k=1}^{R-1}\left(\frac{S_k^2}{k^2}-\frac{\ee|X|^21_{\{|X|\le k\}}}{k}\right)\right|>M/2\right)<\varepsilon/3.
\]
So all we need to show is 
\[
\pp\left(\left|\sum_{k=R}^n\left(\frac{T_k^2}{k^2}
-\frac{\ee|X|^21_{\{|X|\le k\}}}{k}\right)\right|>M/2\right)<\varepsilon/3,
\]
which is implied by
\beq
\label{co}
 \sum_{k=1}^n\frac{T_k^2-b_k}{k^2}=O_\pp(1)
\eeq
with $b_k:=k\ee|X|^21_{\{|X|\le k\}}$. Observe that
\[
  T_k=\sum_{i=1}^k\left( X_{i,k}1_{\{|X_{i,k}|\le k\}}-\ee X_{i,k}1_{\{|X_{i,k}|\le k\}}\right)+c_k=:\tilde{T}_k+c_k,
\]
where $c_k=k\ee X1_{\{|X|\le k\}}$ and
\[
  \sum_{k=1}^n\frac{T_k^2-b_k}{k^2} =
\sum_{k=1}^n\frac{\tilde{T}_k^2-b_k}{k^2}+
\sum_{k=1}^n\frac{c_k^2}{k^2}+
2\sum_{k=1}^n\frac{c_k\tilde{T}_k}{k^2}
=:I_1+I_2+I_3.
\]
Recall that $\ee X=0$ so that
\[
 |c_k|=|k\ee X 1_{\{|X|\le k\}}|=|k\ee X 1_{\{|X|> k\}}|\le\ee |X|^2=1,  
\]
 thus $I_2=O(1)$. We also have $I_3=O_\pp(1)$ because $I_3$ is bounded in $L^2$
\[
  \ee\left(\sum_{k= 1}^n\frac{c_k\tilde{T}_k}{k^2}\right)^2 =
\sum_{k=1}^n\ee\left(\frac{c_k\tilde{T}_k}{k^2}\right)^2
\le\sum_{k=1}^n\frac{1}{k^{4}}k\Var(X1_{\{|X|\le k\}})
\le \ee X^2\sum_{k=1}^n\frac{1}{k^{3}}=O(1).
\]
$I_1$ can be rewritten as
\[
  I_1=\sum_{k=1}^n\frac{\tilde{T}_k^2-\ee \tilde{T}_k^2}{k^2}
-\sum_{k=1}^n\frac{c_k^2}{k^3}
=:I_{11}-I_{12}.
\]
But $I_{12}=O(1)$ since $|c_k|\le 1$. So in order to show \eqref{co} it is enough to 
show that
\[
  K_n:=\sum_{k=1}^n\frac{\tilde{T}_k^2-\ee \tilde{T}_k^2}{k^2}=O_\pp(1),
\]
where $\tilde{T}_k$ is a sum of independent, mean zero rv's with the same distribution as
$X1_{\{|X|\le k\}}-\ee X1_{\{|X|\le k\}}$. This however follows from Lemma \ref{lemma:cor} with $p=4$.
Indeed
\[
  \ee K_n^2 =\Var(K_n)=\sum_{k=1}^n\frac{1}{k^4}\Var(\tilde{T}_k^2)\le
  \ee\sum_{k=1}^\infty\left(\frac{\tilde{T}_k}{k}\right)^4<\infty,
\]
so the proof of \eqref{remainder:1} is complete.
\par To prove \eqref{remainder:2} it suffices to show that
$
  \sum_{k=1}^n\left|\frac{S_k}{k}\right|^3=O_\pp(1)
$
for normalized $X$, which by the same arguments as above is implied by
\[
  \sum_{k=1}^n\left|\frac{T_k}{k}\right|^3=O_\pp(1).
\]
Using the same notation we have $|T_k|^3=|\tilde{T}_k+c_k|^3\le4|\tilde{T}_k|^3+4$. Thus
\[
\sum_{k=1}^n\left|\frac{T_k}{k}\right|^3 \le
4\left(\sum_{k=1}^n\left|\frac{\tilde{T}_k}{k}\right|^3+
\sum_{k=1}^n\frac{1}{k^3}\right)
=:I_4+I_5.
\]
We obviously have $I_5=O(1)$ and by Lemma \ref{lemma:cor} with $p=3$ we get
boundedness of $I_4$ in $L^1$, which completes the proof.
\end{proof}
Now we are in the position to prove the main theorem.
\begin{proof}[Proof of Theorem \ref{theorem1}]
Take $a_n$ and $b_n$ as in the claim and denote
$c_k=\ee|X-\mu|^21_{\{|X-\mu|\le\sigma k\}}$ so $b_n=\sum_{k=1}^n(f(\mu)+f''(\mu)\frac{c_k}{2k})$.
By Taylor's expansion,
\[  f\left(\frac{S_{k,k}}{k}\right)=f(\mu)+f'(\mu)\left(\frac{S_{k,k}}{k}-\mu\right)+\frac{f''(\mu)}{2}\left(\frac{S_{k,k}}{k}-\mu\right)^2
+O\left(\left|\frac{S_{k,k}}{k}-\mu\right|^3\right)\,\text{a.s.},
\]
as a consequence of the SLLN and the assumption of boundedness of $f^{(3)}$ around $\mu$.
Using Lemma \ref{lemma:remainder} we have
\begin{align*}
  \frac{\sum_{k=1}^nf(S_{k,k}/k)-b_n}{a_n}&=
\frac{f'(\mu)}{a_n}\sum_{k=1}^n\left(\frac{S_{k,k}}{k}-\mu\right)+\frac{f''(\mu)}{2a_n}
\sum_{k=1}^n\left[\left(\frac{S_{k,k}}{k}-\mu\right)^2-\frac{c_k}{k}\right]\\
&\qquad+O\left(\frac{1}{a_n}\sum_{k=1}^n\left|\frac{S_{k,k}}{k}-\mu\right|^3\right)\,\text{a.s.}\\
&=\frac{f'(\mu)}{a_n}\sum_{k=1}^n\left(\frac{S_{k,k}}{k}-\mu\right)+o_\pp(1).
\end{align*}
By Lemma \ref{lemma:CLT}
\[
  \frac{f'(\mu)}{a_n}\sum_{k=1}^n\left(\frac{S_{k,k}}{k}-\mu\right)\dto\sigma f'(\mu)\cn,
\]
completing the proof.
\end{proof}
\section{Acknowledgements}
The author is indebted to Dr. R. Lata{\l}a for his insightful comments and supervision which led to an improved
presentation of the paper. Thanks are also due to the anonymous referee for a very careful review of the original
manuscript leading to substantial simplifications.

\end{document}